\documentclass[leqno,12pt]{article} 
\usepackage{amsmath, amssymb}
\usepackage{amsthm} 
\newcommand\supp{\mathop{\rm supp}}
\newcommand\esssup{\mathop{\rm ess \, sup}}
\newcommand\essinf{\mathop{\rm ess \, inf}}

\theoremstyle{plain} 
\newtheorem{theorem}{\indent\sc Theorem}[section]
\newtheorem{lemma}[theorem]{\indent\sc Lemma}

\newtheorem{proposition}[theorem]{\indent\sc Proposition}

\theoremstyle{definition} 
\newtheorem{definition}[theorem]{\indent\sc Definition}
\newtheorem{remark}[theorem]{\indent\sc Remark}

\makeatletter
\def\address#1#2{\begingroup
\noindent\parbox[t]{7.8cm}{%
\small{\scshape\ignorespaces#1}\par\vskip1ex
\noindent\small{\itshape E-mail address}%
\/: #2\par\vskip4ex}\hfill%
\endgroup}%
\makeatother
\title{Convolution operators and variable Hardy spaces on the Heisenberg group} 
\author{
\textsc{Pablo Rocha} 
}
\date{} 
\begin{document}

\maketitle


\begin{abstract}
Let $\mathbb{H}^{n}$ be the Heisenberg group. For $0 \leq \alpha < Q=2n+2$ and $N \in \mathbb{N}$ we consider exponent functions
$p(\cdot) : \mathbb{H}^{n} \to (0, +\infty)$, which satisfy log-H\"older conditions, such that $\frac{Q}{Q+N} < p_{-} \leq p(\cdot) \leq p_{+} < \frac{Q}{\alpha}$. In this article we prove the $H^{p(\cdot)}(\mathbb{H}^{n}) \to L^{q(\cdot)}(\mathbb{H}^{n})$ and 
$H^{p(\cdot)}(\mathbb{H}^{n}) \to H^{q(\cdot)}(\mathbb{H}^{n})$ boundedness of convolution operators with kernels of type 
$(\alpha, N)$ on $\mathbb{H}^{n}$, where $\frac{1}{q(\cdot)} = \frac{1}{p(\cdot)} - \frac{\alpha}{Q}$. In particular, the Riesz potential on $\mathbb{H}^{n}$ satisfies such estimates.

(The proof of Theorem 6.4 was corrected)
\end{abstract}

\section{Introduction}

 On $\mathbb{R}^{n}$, E. Stein and G. Weiss \cite{Weiss} defined the Hardy spaces $H^{p}$, $0 < p < \infty$, by means of the theory of harmonic functions on Euclidean spaces. Later, C. Fefferman and E. Stein \cite{Fefferman} introduced real variable methods into this subject and characterized the Hardy spaces $H^{p}$ in terms of maximal functions. This second approach brought greater flexibility to the whole theory. It is well known that classical Hardy spaces $H^{p}$ with $0 < p \leq 1$ are play an important role in the harmonic analysis.
A remarkable result about Hardy spaces is that every element $f \in H^{p}$, $0 < p \leq 1$, can be expressed of the 
form $f = \sum \lambda_j a_j$, where the $\lambda_j$'s are positive numbers and the $a_j$'s are $p$-atoms (see \cite{Coifman, Latter}). This decomposition allows to study the behavior of certain operators on $H^{p}(\mathbb{R}^{n})$ by focusing one's attention on individual atoms. In principle, the continuity of an operator $T$ on $H^{p}$ can often be proved by estimating $Ta$ when $a(\cdot)$ is a $p$-atom.
Many important operators are better behaved on Hardy spaces $H^{p}$ than on Lebesgue spaces $L^{p}$ in the range $0 < p \leq 1$. For instance, when $p \leq 1$, Riesz transforms on $\mathbb{R}^{n}$ are not bounded on $L^{p}$; however, they are bounded on Hardy spaces $H^{p}$. For more results about Hardy spaces the reader can consult \cite{Coifman2, Taibleson, Stein3, Lu, Uchi, grafakos2}.

G. Folland and E. Stein \cite{Folland} generalized the theory of Hardy spaces $H^{p}$ on homogeneous groups. Two of the main results in this theory are the maximal function characterization of $H^{p}$ and the atomic decomposition of their elements. With this framework, they studied, among others topics, the behavior of convolution operators with kernels of type $(\alpha, N)$ on these spaces.

On the other hand, with the appearing of the theory of variable exponents the Hardy type spaces on $\mathbb{R}^{n}$ received a new impetus 
(see \cite{Orlicz, Kovacik, Diening2, Fiorenza, Nakai, Cruz-Uribe2}). In \cite{Rocha1}, the author jointly with M. Urciuolo proved the 
$H^{p(\cdot)}(\mathbb{R}^{n}) \to L^{q(\cdot)}(\mathbb{R}^{n})$ boundedness of certain generalized Riesz potentials and the 
$H^{p(\cdot)}(\mathbb{R}^{n}) \to H^{q(\cdot)}(\mathbb{R}^{n})$ boundedness of Riesz potentials via the infinite atomic and molecular decomposition developed in \cite{Nakai}. In \cite{Rocha3}, the author gave another proof of the results obtained in \cite{Rocha1}, but by using the finite atomic decomposition given in \cite{Cruz-Uribe2}.

On the Heisenberg group $\mathbb{H}^{n}$, J. Fang and J. Zhao \cite{Fang} gave a variety of distinct approaches, based on differing definitions, all lead to the same notion of variable Hardy space $H^{p(\cdot)}(\mathbb{H}^{n})$. One of their main goals is the atomic decomposition of elements in $H^{p(\cdot)}(\mathbb{H}^{n})$, as an application of the atomic decomposition they proved that singular integrals are bounded on $H^{p(\cdot)}(\mathbb{H}^{n})$.

Let $0 \leq \alpha < Q:=2n+2$ and $N \in \mathbb{N}$. For $0 < \alpha < Q$, a function $K_{\alpha} \in C^{N}(\mathbb{H}^{n} \setminus \{ e \})$ is said to be a kernel of type $(\alpha, N)$ on $\mathbb{H}^{n}$ if
\begin{equation} \label{decay0}
\left|(\widetilde{X}^{I}K_{\alpha})(z) \right| \lesssim \rho(z)^{\alpha - Q - d(I)} \,\,\,\, \text{for all} \,\, d(I) \leq N \,\,\, \text{and all} \,\, z \neq e,
\end{equation}
where $\widetilde{X}^{I}$ is the right-invariant higher order derivative associated to the multiindex $I = (i_1, ..., i_{2n}, i_{2n+1})$,
$d(I) = i_1 + \cdot \cdot \cdot i_{2n} + 2 i_{2n+1}$, and $\rho(\cdot)$ is the \textit{Koranyi norm} on $\mathbb{H}^{n}$ given by 
(\ref{Koranyi norm}). A distribution $K_{0}$ is said to be a kernel of type $(0,N)$ on $\mathbb{H}^{n}$ if is of class $C^{N}$ on 
$\mathbb{H}^{n} \setminus \{ e \}$, satisfies (\ref{decay0}) with $\alpha = 0$, and $\| f \ast K_0 \|_{2} \leq \| f \|_2$ for all 
$f \in \mathcal{S}(\mathbb{H}^{n})$.

The purpose of this work is to generalize \cite[Theorem 6.10]{Folland} to the context of variable exponents when 
$G = \mathbb{H}^{n}$. More precisely, we will prove that the operators defined by right convolution with kernels of type $(\alpha, N)$ 
on $\mathbb{H}^{n}$ can be extended to bounded operators $H^{p(\cdot)}(\mathbb{H}^{n}) \to L^{q(\cdot)}(\mathbb{H}^{n})$ and 
$H^{p(\cdot)}(\mathbb{H}^{n}) \to H^{q(\cdot)}(\mathbb{H}^{n})$ for certain variable exponents $p(\cdot)$ and $q(\cdot)$ related by 
$\frac{1}{p(\cdot)} - \frac{1}{q(\cdot)} = \frac{\alpha}{Q}$ with $0 \leq \alpha < Q$ (see Theorems \ref{Hp-Lq} and \ref{Hp-Hq} below). As 
an application of these results we obtain that Riesz potential $\mathcal{R}_{\alpha}$ on $\mathbb{H}^{n}$ admits such extensions (see Theorem \ref{Riesz estimates}).

\

This paper is organized as follows. Section 2 presents the basics about the Heisenberg group $\mathbb{H}^{n}$ and some properties of 
variable Lebesgue spaces $L^{p(\cdot)}\left(\mathbb{H}^{n}\right)$. In Section 3 we state three auxiliary results, two of them referring to the $L^{p(\cdot)}\left(\mathbb{H}^{n}\right)$ - norm of the characteristic functions of balls in $\mathbb{H}^{n}$ and the other one is a supporting result. In Section 4, we establish the off-diagonal version of the Fefferman-Stein vector-valued maximal inequality for the fractional maximal operator on $\mathbb{H}^{n}$ in the context of variable Lebesgue spaces, this result is crucial to get the main goals of Section 6. In Section 5, we recall the definition and the atomic decomposition of variable Hardy spaces on $\mathbb{H}^{n}$ given in 
\cite{Fang}. Finally, in Section 6 we prove our main results.

\

\textbf{Notation:} The symbol $A\lesssim B$ stands for the inequality $A \leq cB$ for some positive constant $c$. The symbol $A \approx B$ stands for $B \lesssim A \lesssim B$. For a measurable subset $E\subseteq \mathbb{H}^{n}$ we denote by $\left\vert E\right\vert $ and 
$\chi_{E}$ the Haar measure of $E$ and the characteristic function of $E$ respectively.

\section{Preliminaries} 

Let $J$ be the $2n \times 2n$ skew-symmetric matrix given by
\[
J= \frac{1}{2} \left( \begin{array}{cc}
                           0 & -I_n \\
                           I_n & 0 \\
                                      \end{array} \right)
\]
where $I_n$ is the $n \times n$ identity matrix.

The  Heisenberg group $\mathbb{H}^{n}$ is a homogeneous group whose underlying manifold is  $\mathbb{R}^{2n} \times \mathbb{R}$
(see \cite{Folland, Fischer}). This is, $\mathbb{H}^{n}$ can be identified with $\mathbb{R}^{2n} \times \mathbb{R}$ with group law 
(noncommutative) given by
\[
(x,t) \cdot (y,s) = \left( x+y, t+s + x^{t} J y \right)
\]
and dilations 
\[
r \cdot (x,t) = (rx, r^{2}t), \,\,\,\ r > 0.
\]
With this structure we have that $e = (0,0)$ is the neutral element, $(x, t)^{-1}=(-x, -t)$ is the inverse of $(x, t)$, and
$r \cdot((x,t) \cdot (y,s)) = (r\cdot(x,y)) \cdot (r\cdot(y,s))$. The topology in $\mathbb{H}^{n}$ is the induced by $\mathbb{R}^{2n} \times \mathbb{R} \equiv \mathbb{R}^{2n+1}$, so the borelian sets of $\mathbb{H}^{n}$ are identified with those of $\mathbb{R}^{2n+1}$. The Haar measure in $\mathbb{H}^{n}$ is the Lebesgue measure of $\mathbb{R}^{2n+1}$, thus $L^{p}(\mathbb{H}^{n}) \equiv L^{p}(\mathbb{R}^{2n+1})$, 
$0 < p \leq \infty$. Moreover, for $f \in L^{1}(\mathbb{H}^{n})$ and each $w \in \mathbb{H}^{n}$
\begin{equation} \label{invariant transl}
\int_{\mathbb{H}^{n}} f(w \cdot z) \, dz = \int_{\mathbb{H}^{n}} f(z \cdot w) \, dz = \int_{\mathbb{H}^{n}} f(z) \, dz, 
\end{equation}
for $r > 0$ fixed, we also have
\[
\int_{\mathbb{H}^{n}} f(r \cdot z) \, dz = r^{-Q} \int_{\mathbb{H}^{n}} f(z) \, dz,
\]
where $Q= 2n+2$. The number $2n+2$ is known as the {\it homogeneous dimension} of $\mathbb{H}^{n}$ (we observe that the {\it topological dimension} of $\mathbb{H}^{n}$ is $2n+1$).

The {\it Koranyi norm} on $\mathbb{H}^{n}$ is the function $\rho : \mathbb{H}^{n} \to [0, \infty)$ defined by
\begin{equation} \label{Koranyi norm}
\rho(x,t) = \left( |x|^{4} + 16 \, t^{2}  \right)^{1/4}, \,\,\, (x,t) \in \mathbb{H}^{n},
\end{equation}
where $| \cdot |$ is the usual Euclidean norm on $\mathbb{R}^{2n}$. Let $z = (x,t)$ and $w = (y,s) \in \mathbb{H}^{n}$, the Koranyi norm satisfies the following properties
\begin{eqnarray*}
\rho(z) & = & 0 \,\,\, \text{if and only if} \,\, z = e, \\
\rho(z^{-1}) & = & \rho(z) \,\,\,\, \text{for all} \,\, z \in \mathbb{H}^{n}, \\
\rho(r \cdot z) & = & r \rho(z) \,\,\,\, \text{for all} \,\, z \in \mathbb{H}^{n} \,\, \text{and all} \,\, r > 0, \\
\rho(z \cdot w) & \leq & \rho(z) + \rho(w) \,\,\,\, \text{for all} \,\, z, w \in \mathbb{H}^{n}, \\
| \rho(z) - \rho(w) | & \leq & \rho(z \cdot w) \,\,\,\, \text{for all} \,\, z, w \in \mathbb{H}^{n}.
\end{eqnarray*}
Moreover, $\rho$ is continuous on $\mathbb{H}^{n}$ and is smooth on $\mathbb{H}^{n} \setminus \{ e \}$. The $\rho$ - ball centered at 
$z_0 \in \mathbb{H}^{n}$ with radius $\delta > 0$ is defined by
\[
B(z_0, \delta) := \{ w \in \mathbb{H}^{n} : \rho(z_0^{-1} w) < \delta \}.
\]

\begin{remark}
The topology in $\mathbb{H}^{n}$ induced by the $\rho$ - balls coincides with the Euclidean topology of $\mathbb{R}^{2n+1}$ (see
\cite[Proposition 3.1.37]{Fischer}).
\end{remark}

Let $|B(z_0, \delta)|$ be the Haar measure of the $\rho$ - ball $B(z_0, \delta) \subset \mathbb{H}^{n}$. Then, 
\[
|B(z_0, \delta)| = c \delta^{Q},
\]
where $c = |B(e,1)|$ and  $Q = 2n+2$. Given $\lambda > 0$, we put $\lambda B = \lambda B(z_0, \delta) = 
B(z_0, \lambda \delta)$. So $|\lambda B| = \lambda^{Q}	|B|$.

\begin{remark}
For any $z, z_0 \in \mathbb{H}^{n}$ and $\delta >0$, we have
\[
z_0 \cdot B(z, \delta) = B(z_0 z, \delta).
\]
In particular, $B(z, \delta) = z \cdot B(e, \delta)$. It is also easy to check that $B(e, \delta) = \delta \cdot B(e,1)$ for any $\delta > 0$.
\end{remark}
\begin{remark} \label{cambio de centro}
If $f \in L^{1}(\mathbb{H}^{n})$, then for every $\rho$ - ball $B$ and every $z_0 \in \mathbb{H}^{n}$, by (\ref{invariant transl}), we have
\[
\int_{B} f(z) \, dz = \int_{z_{0}^{-1} \cdot B} f(z_0 \cdot u) \, du.
\]
\end{remark}

If $f$ and $g$ are measurable functions on $\mathbb{H}^{n}$, their convolution $f * g$ is defined by
\[
(f * g)(z) := \int_{\mathbb{H}^{n}} f(w) g(w^{-1} \cdot z) \, dw,
\]
when the integral is finite.

For every $i = 1,2, ..., 2n+1$, $X_i$ denotes the left invariant vector field which is defined by
\[
(X_i f)(x,t) = \frac{d}{ds} f((x,t) \cdot s e_i) |_{s=0},
\]
where $\{ e_i \}_{i=1}^{2n+1}$ is the canonical basis of $\mathbb{R}^{2n+1}$. Thus
\[
X_i = \frac{\partial}{\partial x_i} + \frac{x_{i+n}}{2} \frac{\partial}{\partial t}, \,\,\,\, i=1, 2, ..., n;
\]
\[
X_{i+n} = \frac{\partial}{\partial x_{i+n}} - \frac{x_{i}}{2} \frac{\partial}{\partial t}, \,\,\, i=1, 2, ..., n;
\]
and
\[
X_{2n+1} = \frac{\partial}{\partial t}.
\]

Similarly, we define the right invariant vector fields $\{ \widetilde{X}_i \}_{i=1}^{2n+1}$ by
\[
(\widetilde{X}_i f)(x,t) = \frac{d}{ds} f(s e_i \cdot (x,t)) |_{s=0}.
\]
Then
\[
\widetilde{X}_i = \frac{\partial}{\partial x_i} - \frac{x_{i+n}}{2} \frac{\partial}{\partial t}, \,\,\,\, i=1, 2, ..., n;
\]
\[
\widetilde{X}_{i+n} = \frac{\partial}{\partial x_{i+n}} + \frac{x_{i}}{2} \frac{\partial}{\partial t}, \,\,\, i=1, 2, ..., n;
\]
and
\[
\widetilde{X}_{2n+1} = \frac{\partial}{\partial t}.
\]

Given a multiindex $I=(i_1,i_2, ..., i_{2n}, i_{2n+1}) \in (\mathbb{N} \cup \{ 0 \})^{2n+1}$, we set
\[
|I| = i_1 + i_2 + \cdot \cdot \cdot + i_{2n} + i_{2n+1}, \hspace{.5cm} d(I) = i_1 + i_2 + \cdot \cdot \cdot + i_{2n} + 2 \, i_{2n+1}.
\]
The amount $|I|$ is called the length of $I$ and $d(I)$ the homogeneous degree of $I$. We adopt the following multiindex notation for
higher order derivatives and for monomials on $\mathbb{H}^{n}$. If $I=(i_1, i_2, ..., i_{2n+1})$ is a multiindex, 
$X = \{ X_i \}_{i=1}^{2n+1}$, $\widetilde{X} =  \{ \widetilde{X}_{i} \}_{i=1}^{2n+1}$, and 
$z = (x,t) = (x_1, ..., x_{2n}, t) \in \mathbb{H}^{n}$, we put
\[
X^{I} := X_{1}^{i_1} X_{2}^{i_2} \cdot \cdot \cdot X_{2n+1}^{i_{2n+1}}, \,\,\,\,\,\, 
\widetilde{X}^{I} := \widetilde{X}_{1}^{i_1} \widetilde{X}_{2}^{i_2} \cdot \cdot \cdot \widetilde{X}_{2n+1}^{i_{2n+1}},
\]
and
\[
z^{I} := x_{1}^{i_1} \cdot \cdot \cdot x_{2n}^{i_{2n}} \cdot t^{i_{2n+1}}.
\]
A computation give
\[
X^{I}(f(r \cdot z)) = r^{d(I)} (X^{I}f)(r\cdot z), \,\,\,\,\,\, \widetilde{X}^{I}(f(r \cdot z)) = r^{d(I)} (\widetilde{X}^{I}f)(r\cdot z)
\]
and 
\[
(r\cdot z)^{I} = r^{d(I)} z^{I}.
\]
So, the operators $X^{I}$ and $\widetilde{X}^{I}$ and the monomials $z^{I}$ are homogeneous of degree $d(I)$. The operators $X^{I}$ and 
$\widetilde{X}^{I}$ interact with the convolutions in the following way
\[
X^{I}(f \ast g) = f \ast (X^{I}g), \,\,\,\,\,\, \widetilde{X}^{I}(f \ast g) = (\widetilde{X}^{I} f) \ast g, \,\,\, \text{and} \,\,\,
(X^{I} f) \ast g = f \ast (\widetilde{X}^{I} g).
\]

The Schwartz space $\mathcal{S}(\mathbb{H}^{n})$ is defined by
\[
\mathcal{S}(\mathbb{H}^{n}) = \left\{ \phi \in C^{\infty}(\mathbb{H}^{n}) : \sup_{z \in \mathbb{H}^{n}} (1+\rho(z))^{L} 
|(X^{I}\phi)(z)| < \infty \,\,\, \forall \,\, L \in \mathbb{N}_{0}, \, I \in (\mathbb{N}_{0})^{2n+1}  \right\}.
\]
We topologize the space $\mathcal{S}(\mathbb{H}^{n})$ with the following family of seminorms
\[
\| \phi \|_{\mathcal{S}(\mathbb{H}^{n}), \, L} = \sum_{d(I) \leq L} \sup_{z \in \mathbb{H}^{n}} (1+\rho(z))^{(L+1)(Q+1)} |(X^{I}\phi)(z)| \,\,\,\,\,\,\, 
(L \in \mathbb{N}_{0}),
\]
with $\mathcal{S}'(\mathbb{H}^{n})$ we denote the dual space of $\mathcal{S}(\mathbb{H}^{n})$. 

\

Now, we briefly present the basics of variable Lebesgue spaces. Let $p(\cdot) : \mathbb{H}^{n} \to (0, \infty)$ be a measurable function. Given a measurable set $E \subset \mathbb{H}^{n}$, let
\[
p_{-}(E) = \essinf_{ z \in E } p(z), \,\,\,\, \text{and} \,\,\,\, p_{+}(E) = \esssup_{z \in E} p(z).
\]
When $E = \mathbb{H}^{n}$, we will simply write $p_{-} := p_{-}(\mathbb{H}^{n})$, $p_{+} := p_{+}(\mathbb{H}^{n})$ and 
$\underline{p} := \min \{ p_{-}, 1 \}$. Such function $p(\cdot)$ is called an exponent function.

We define the variable Lebesgue space $L^{p(\cdot)} = L^{p(\cdot)}(\mathbb{H}^{n})$ to be the set of all measurable functions 
$f: \mathbb{H}^{n} \to \mathbb{C}$ such that for some $\lambda > 0$ 
\[
\int_{\mathbb{H}^{n}} |f(z)/\lambda|^{p(z)} dz < \infty.
\]
This becomes a quasi normed space when equipped with the Luxemburg norm
\[
\| f \|_{L^{p(\cdot)}} = \inf \left\{ \lambda > 0 : \int_{\mathbb{H}^{n}} |f(z)/\lambda|^{p(z)} dz \leq 1 \right\}.
\]
The following result follows from the definition of the $L^{p(\cdot)}$ - norm.

\begin{lemma} \label{potencia s}
Given a measurable function $p(\cdot) : \mathbb{H}^{n} \to (0, \infty)$ with $0 < p_{-} \leq p_{+} < \infty$, then 
\\
(i) $\| f \|_{L^{p(\cdot)}} \geq 0$ and $\| f \|_{L^{p(\cdot)}} = 0$ if and only if $f \equiv 0$ a.e.,
\\
(ii) $\| c f \|_{L^{p(\cdot)}} = |c| \| f \|_{L^{p(\cdot)}}$ for all $f \in L^{p(\cdot)}$ and all $c \in \mathbb{C}$,
\\
(iii) $\| f +  g \|_{L^{p(\cdot)}} \leq 2^{1/\underline{p} - 1}(\| f \|_{L^{p(\cdot)}} + \| g \|_{L^{p(\cdot)}})$ for all 
$f, g \in L^{p(\cdot)}$,
\\
(iv) $\| f \|_{L^{p(\cdot)}}^{s} = \| |f|^{s} \|_{L^{p(\cdot)/s}}$ for every $s > 0$.
\end{lemma}

For an exponent function $p(\cdot) : \mathbb{H}^{n} \to (1, \infty)$, its conjugate function $p'(\cdot)$ is defined by 
$\frac{1}{p(z)} + \frac{1}{p'(z)} = 1$. A straightforward computation shows that
\[
(p'(\cdot))_{+} = (p_{-})', \,\,\,\,\, \text{and} \,\,\,\,\,  (p'(\cdot))_{-} = (p_{+})'.
\]

We have the following generalization of H\"older's inequality and an equivalent expression for the $L^{p(\cdot)}$ - norm.

\begin{lemma} (H\"older's inequality) \label{Holder ineq}
Let $p(\cdot) : \mathbb{H}^{n} \to (1, \infty)$ be a measurable function and let $p'(\cdot)$ be its conjugate function. Then, 
\[
\int_{\mathbb{H}^{n}} |f(z) g(z)| dz \leq 2 \| f \|_{L^{p(\cdot)}} \| g \|_{L^{p'(\cdot)}}.
\]
\end{lemma}

\begin{proof} The lemma follows from \cite[Lemma 3.2.20]{Diening2}.
\end{proof}

\begin{proposition} \label{norma equivalente}
Let $p(\cdot) : \mathbb{H}^{n} \to (1, \infty)$ be a measurable function and let $p'(\cdot)$ be its conjugate function. Then
\[
\| f \|_{L^{p(\cdot)}} \approx \sup \left\{ \int_{\mathbb{H}^{n}} |f(z) g(z)| dz : \| g \|_{L^{p'(\cdot)}} \leq 1
\right\}.
\]
\end{proposition}

\begin{proof} The proposition follows from \cite[Corollary 3.2.14]{Diening2}.
\end{proof}

We say that an exponent function $p(\cdot) : \mathbb{H}^{n} \to (0, \infty)$ such that $0 < p_{-} \leq p_{+} < \infty$ belongs 
to $\mathcal{P}^{\log}(\mathbb{H}^{n})$, if there exist positive constants $C$, $C_{\infty}$ and $p_{\infty}$ such that $p(\cdot)$ satisfies the local log-H\"older continuity condition, i.e.:
\[
|p(z) - p(w)| \leq \frac{C}{-\log(\rho(z^{-1}w))}, \,\,\, \text{for} \,\, \rho(z^{-1}w) \leq \frac{1}{2},
\]
and is log-H\"older continuous at infinity, i.e.:
\[
|p(z) - p_{\infty}| \leq \frac{C_{\infty}}{\log(e+\rho(z))}, \,\,\, \text{for all} \,\, z \in \mathbb{H}^{n}.
\]
Here $\rho$ is the {\it Koranyi norm} given by (\ref{Koranyi norm}).

\begin{lemma} \label{exp prop}
Let $p : \mathbb{H}^{n} \to (0, \infty)$ be an exponent function. Then

$(i)$ if $1 < p_{-} \leq p_{+} < \infty$, then $p(\cdot) \in \mathcal{P}^{\log}(\mathbb{H}^{n})$ if and only if
$p'(\cdot) \in \mathcal{P}^{\log}(\mathbb{H}^{n})$, where $(p_{\infty})' = (p')_{\infty}$;

$(ii)$ if $0 < p_{-} \leq p_{+} < \infty$, then $p(\cdot) \in \mathcal{P}^{\log}(\mathbb{H}^{n})$ if and only if
$\frac{1}{p(\cdot)} \in \mathcal{P}^{\log}(\mathbb{H}^{n})$.
\end{lemma}

\begin{proof} The statement $(i)$ is obvious. Now, $(ii)$ follows from the following inequality valid for all $z,w \in \mathbb{H}^{n}$
\[
\left| \frac{p(z) - p(w)}{(p_{+})^{2}} \right| \leq \left| \frac{1}{p(z)} - \frac{1}{p(w)} \right| \leq 
\left| \frac{p(z) - p(w)}{(p_{-})^{2}} \right|.
\]
\end{proof}

\section{Auxiliary results}

The following three results are crucial to get the main results of Section 6. The first two talk about the size of the $\rho$ - balls in the
$L^{p(\cdot)}$ - norm, and the last one is a supporting result.

\begin{lemma} \label{estim pp prime}
Let $p(\cdot) \in \mathcal{P}^{\log}(\mathbb{H}^{n})$ with $1 < p_{-} \leq p_{+} < \infty$. Then
\[
\| \chi_B \|_{L^{p(\cdot)}(\mathbb{H}^{n})} \| \chi_B \|_{L^{p'(\cdot)}(\mathbb{H}^{n})} \approx |B|
\]
uniformly for all $\rho$ - balls $B \subset \mathbb{H}^{n}$.
\end{lemma}

\begin{proof}
By Lemma \ref{exp prop} - (i) we have that $p'(\cdot) \in \mathcal{P}^{\log}(\mathbb{H}^{n})$ with $(p')_{\infty} = (p_{\infty})'$, since 
$p(\cdot) \in \mathcal{P}^{\log}(\mathbb{H}^{n})$. Now, the lemma follows from \cite[Lemma 4.1]{Fang}.
\end{proof}

\begin{lemma} \label{2B}
Let $p(\cdot) \in \mathcal{P}^{\log}(\mathbb{H}^{n})$ with $1 < p_{-} \leq p_{+} < \infty$ and let $\lambda > 1$ be fixed. Then
\[
\| \chi_{\lambda B} \|_{L^{p(\cdot)}(\mathbb{H}^{n})} \approx \| \chi_B \|_{L^{p(\cdot)}(\mathbb{H}^{n})}
\]
uniformly for all $\rho$ - balls $B \subset \mathbb{H}^{n}$.
\end{lemma}

\begin{proof}
By the order preserving property of the norm $\| \cdot \|_{L^{p(\cdot)}}$ we have that
\begin{equation} \label{estim B2B}
\| \chi_{B} \|_{L^{p(\cdot)}} \leq \| \chi_{\lambda B} \|_{L^{p(\cdot)}}.
\end{equation}
On the other hand, by Lemma \ref{estim pp prime}, (\ref{estim B2B}) above, and H\"older's inequality applied to $|B| = \int \chi_B(z) dz$, result
\[
\| \chi_{\lambda B} \|_{L^{p(\cdot)}} \leq C_{\lambda} |B| \| \chi_{\lambda B} \|_{L^{p'(\cdot)}}^{-1} \leq C_{\lambda} |B| 
\| \chi_{B} \|_{L^{p'(\cdot)}}^{-1} \leq C_{\lambda}  \| \chi_{B} \|_{L^{p(\cdot)}}.
\]
This completes the proof.
\end{proof}

The following result is an adaptation of \cite[Lemma 5.4]{Ho1} to our setting.

\begin{proposition} \label{b_k functions}
Let $q(\cdot) : \mathbb{H}^{n} \to (0, \infty)$ such that $q(\cdot) \in \mathcal{P}^{\log}(\mathbb{H}^{n})$ and 
$0 < q_{-} \leq q_{+} < \infty$. Let $s > 1$ and $0 < q_{*} < \underline{q}$ such that $s q_{*} > q_{+}$ and let $\{ b_k \}_{k=1}^{\infty}$ be a sequence of nonnegative functions in $L^{s}(\mathbb{H}^{n})$ such that each $b_k$ is supported in a $\rho$ - ball 
$B_k \subset \mathbb{H}^{n}$ and
\begin{equation} \label{bks}
\| b_k \|_{L^{s}(\mathbb{H}^{n})} \leq A_k |B_k|^{1/s},
\end{equation}
where $A_k >0$ for all $k \geq 1$. Then, for any sequence of nonnegative numbers $\{ \lambda_k \}_{k=1}^{\infty}$ we have
\[
\left\| \sum_{k=1}^{\infty} \lambda_k b_k \right\|_{L^{q(\cdot)/q_{*}}(\mathbb{H}^{n})} \leq C \left\| \sum_{k=1}^{\infty} A_k \lambda_k \chi_{B_k} \right\|_{L^{q(\cdot)/q_{*}}(\mathbb{H}^{n})},
\]
where $C$ is a positive constant which does not depend on $\{ b_k \}_{k=1}^{\infty}$, $\{ A_k \}_{k=1}^{\infty}$, and 
$\{ \lambda_k \}_{k=1}^{\infty}$.
\end{proposition}

\begin{proof}
Given $g \in L^{1}_{loc}(\mathbb{H}^{n})$, by (\ref{bks}) and H\"older's inequality, we have 
\[
\int_{\mathbb{H}^{n}} b_k(z) |g(z)| dz \leq \| b_k \|_{L^{s}} \| \chi_{B_k} \, g \|_{L^{s'}}
\leq A_k |B_k|^{1/s} \left( \int_{B_k} |g(w)|^{s'} dw \right)^{1/s'}
\]
\[
= A_k |B_k| \left( \frac{1}{|B_k|} \int_{B_k} |g(w)|^{s'} dw \right)^{1/s'}
\]
\[
= A_k \int_{\mathbb{H}^{n}} \left( \frac{1}{|B_k|} \int_{B_k} |g(w)|^{s'} dw \right)^{1/s'} \chi_{B_k}(z) dz
\]
\[
\leq A_k \int_{B_k} \left[ M(|g|^{s'})(z) \right]^{1/s'} dz.
\]
So
\begin{equation} \label{series estimate}
\int_{\mathbb{H}^{n}} \left( \sum_{k} \lambda_k b_k (z) \right) |g(z)| dz \leq \sum_{k} A_k \lambda_k \int_{B_k} 
\left[ M(|g|^{s'})(z) \right]^{1/s'} dz
\end{equation}
\[
= \int_{\mathbb{H}^{n}} \left( \sum_{k} A_k \lambda_k \chi_{B_k}(z) \right) \left[ M(|g|^{s'})(z) \right]^{1/s'} dz,
\]
\[
\lesssim \left\| \sum_{k} A_k \lambda_k \chi_{B_k} \right\|_{L^{q(\cdot)/q_{*}}} 
\left\| M(|g|^{s'}) \right\|_{L^{(q(\cdot)/q_{*})'/s'}}^{1/s'},
\]
where the second inequality follows from Lemmas \ref{Holder ineq} and \ref{potencia s}. Now, it is clear that 
\[
1 < \frac{q_{-}}{q_{*}} = \left(\frac{q(\cdot)}{q_{*}} \right)_{-} \leq \frac{q(\cdot)}{q_{*}} \leq
\left(\frac{q(\cdot)}{q_{*}} \right)_{+} = \frac{q_{+}}{q_{*}} < s,
\] 
and so
\[
1 < s' < \left( \left(\frac{q(\cdot)}{q_{*}} \right)' \right)_{-} \leq \left(\frac{q(\cdot)}{q_{*}} \right)' \leq
\left( \left(\frac{q(\cdot)}{q_{*}} \right)' \right)_{+} = \frac{q_{-}}{q_{-} - q_{*}} < \infty.
\]
Since $q(\cdot)/q_{*} \in \mathcal{P}^{\log}(\mathbb{H}^{n})$ we have that 
$\left( q(\cdot)/q_{*} \right)'/s' \in \mathcal{P}^{\log}(\mathbb{H}^{n})$ with 
$(\left( q(\cdot)/q_{*} \right)')_{-}/s' > 1$. Then, by Lemma \ref{exp prop} - (ii), \cite[Theorem 1.4 and 1.7]{Adam}, \cite[Chapter I, 2.5 and Theorem 1]{Stein3}, 
(\ref{series estimate}) and Lemma \ref{potencia s}, it follows that
\begin{equation} \label{series estimate 2}
\int_{\mathbb{H}^{n}} \left( \sum_{k} \lambda_k b_k (z) \right) |g(z)| dz
\lesssim \left\| \sum_{k} A_k \lambda_k \chi_{B_k} \right\|_{L^{q(\cdot)/q_{*}}} 
\left\| g \right\|_{L^{(q(\cdot)/q_{*})'}},
\end{equation}
for all $g \in L^{(q(\cdot)/q_{*})'}$. Finally, by taking the supremum over all $g$ with $\left\| g \right\|_{L^{(q(\cdot)/q_{*})'}} \leq 1$ in (\ref{series estimate 2}), the proposition follows from Proposition \ref{norma equivalente}.
\end{proof}

\begin{remark}
Proposition \ref{b_k functions} still holds if one considers a sequence of complex functions $b_k : \mathbb{H}^{n} \to \mathbb{C}$ supported 
on $\rho$ - balls $B_k$ where the family $\{ B_k \}$ of all these balls satisfies the bounded intersection property.
\end{remark}

\section{Fractional maximal operator}

We recall that the homogeneous dimension of $\mathbb{H}^{n}$ is $Q = 2n+2$. For $0 < \alpha < Q$, we define the fractional maximal 
operator $M_{\alpha}$ by
\[
M_{\alpha}f(z) = \sup_{B \ni z} |B|^{\frac{\alpha}{Q} - 1}\int_{B} |f(w)| \, dw,
\]
where $f$ is a locally integrable function on $\mathbb{H}^{n}$ and the supremum is taken over all the $\rho$ - balls $B$ containing $z$. For $\alpha=0$, we have that $M_0 = M$, where $M$ is the {\it Hardy-Littlewood maximal operator} on $\mathbb{H}^{n}$. 

A measurable function $\omega : \mathbb{H}^{n} \to \mathbb{R}$ is called a weight if $\omega(z) > 0$ a.e. $z \in \mathbb{H}^{n}$ and 
$\omega$ is locally
integrable.

Let $ p \in \mathbb{R} \setminus \{ 0 \}$ and $0 < s < \infty$. Given a weight $\omega$ and a measurable set $E \subset \mathbb{H}^{n}$, we write
\[
[\omega^{p}(E)]^{s} = \left( \int_{E} [\omega(z)]^{p} \, dz \right)^{s}.
\] 

We say that a weight $\omega$ belongs to the class $\mathcal{A}_1$ if there exists a positive constant $C$ such that
\[
(M \omega)(z) \leq C \omega(z), \,\,\, \text{a.e.} \,\, z \in \mathbb{H}^{n}.
\]

Given a weight $\omega$ and $p > 1$, set $\sigma := \omega^{-1/(p-1)}$. We say that $\omega$ belongs to the class $\mathcal{A}_p$ if 
\[
[\omega]_{\mathcal{A}_p} : = \sup_B \frac{\omega(B) [\sigma(B)]^{p-1}}{|B|^{p}} < \infty,
\]
where the supremum is taken over all the $\rho$ - balls $B$ of $\mathbb{H}^{n}$.

A weight $\omega$ satisfies the {\it reverse doubling condition} (RD) if there exist $0 < \alpha, \beta < 1$ such that whenever 
$|B(z, \delta_1)| \leq \alpha |B(z, \delta_2)|$, $0 < \delta_1 < \delta_2$, we have 
$\omega(B(z, \delta_1)) \leq \beta \omega(B(z, \delta_2))$.

The proofs of the following two lemmas are analogous to that of Euclidean case, so we omit the proofs.

\begin{lemma} \label{prop Ap}
Let $1 \leq s < p < \infty$ and $0 < r < 1$. Then

(i) if $\omega \in \mathcal{A}_1$, then $\omega^{r} \in \mathcal{A}_1$;

(ii)  $\omega \in \mathcal{A}_p$ if and only if $\sigma \in \mathcal{A}_{p'}$, where $\frac{1}{p} + \frac{1}{p'} =1$;

(iii) $\mathcal{A}_{s} \subset \mathcal{A}_{p}$;

(iv) if $\omega \in \mathcal{A}_s$, then $\omega$ satisfies the (RD) condition.
\end{lemma}

Given $1 < p \leq q < \infty$, we say that a weight $\omega$ belongs to the class $\mathcal{A}_{p, q}$ if 
\begin{equation} \label{cte Apq}
[\omega]_{\mathcal{A}_{p, q}} := \sup_B \frac{[\omega^{q}(B)]^{\frac{1}{q}} [\omega^{-p'}(B)]^{\frac{1}{p'}}}{|B|^{\frac{1}{q} + \frac{1}{p'}}} < \infty,
\end{equation}
where the supremum is taken over all the $\rho$ - balls $B$ of $\mathbb{H}^{n}$.

\begin{lemma} \label{prop Apq}
Let $1 < p \leq q < \infty$, then

(i) $\omega \in \mathcal{A}_{p, q}$ if and only if $\omega^{q} \in \mathcal{A}_{1+ q/p'}$;

(ii) $\omega \in \mathcal{A}_{p, q}$ if and only if $\omega^{-1} \in \mathcal{A}_{q', p'}$;

(iii) $\omega \in \mathcal{A}_{p, q}$ if and only if $\omega^{-p'} \in \mathcal{A}_{1+ p'/q}$;

(iv) if $\omega \in \mathcal{A}_{1}$, then $\omega^{1/q} \in \mathcal{A}_{p, q}$.
\end{lemma}

\begin{proposition} \label{fractional max}
Let $0 < \alpha < Q$, $1 < p < \frac{Q}{\alpha}$ and $\frac{1}{q} = \frac{1}{p} - \frac{\alpha}{Q}$. If $\omega \in \mathcal{A}_{p,q}$, then
\[
\left(\int_{\mathbb{H}^{n}} [M_{\alpha}f(z)]^{q} [\omega(z)]^{q} dz \right)^{1/q} \leq C 
\left(\int_{\mathbb{H}^{n}} |f(z)|^{p} [\omega(z)]^{p} dz \right)^{1/p}.
\]
\end{proposition}

\begin{proof}
Let $1 < p < \frac{Q}{\alpha}$, it is easy to check that $-p' = p(1-p')$. If $\omega \in \mathcal{A}_{p,q}$, from 
Lemma \ref{prop Apq}-(iii), we have $\omega^{p(1-p')} \in \mathcal{A}_{1+ p'/q}$. From Lemma \ref{prop Ap}-(iv) it follows that
$\omega^{p(1-p')}$ satisfies the (RD) condition. So, by taking into account that $\omega$ satisfies (\ref{cte Apq}), the proposition follows to apply \cite[Theorem 3.1]{Goga} with $\gamma = \frac{\alpha}{Q}$, $\omega^{p(1-p')}$ instead of $\omega^{1-p'}$, $\nu = \omega^{q}$, and 
$\mu$ being the Haar measure on $\mathbb{H}^{n}$.
\end{proof}

The following result is an "off-diagonal" version of the Fefferman-Stein vector-valued maximal inequality for the fractional maximal operator on Heisenberg group. 

\begin{theorem} \label{Feff-Stein ineq}
Let $0 \leq \alpha < Q$, $1 < r < \infty$, and let $p(\cdot) \in \mathcal{P}^{\log}(\mathbb{H}^{n})$ with 
$1 < p_{-} \leq p_{+} < \frac{Q}{\alpha}$. If $\frac{1}{q(\cdot)} := \frac{1}{p(\cdot)} - \frac{\alpha}{Q}$, then
\begin{equation} \label{maximal fract ineq}
\left\| \left( \sum_{j=1}^{\infty} (M_{\alpha} f_j)^{r} \right)^{1/r} \right\|_{L^{q(\cdot)} (\mathbb{H}^{n})} \lesssim 
\left\| \left( \sum_{j=1}^{\infty} |f_j|^{r} \right)^{1/r} \right\|_{L^{p(\cdot)} (\mathbb{H}^{n})},
\end{equation}
holds for all sequences of bounded measurable functions with compact support $\{ f_j \}_{j=1}^{\infty}$.
\end{theorem}

\begin{proof}
The case $\alpha = 0$ was proved in \cite[see Theorem 4.2]{Fang}. For the case $0 < \alpha < Q$, we define
$$\mathcal{F} = \left\{ \left( \left( \sum_{j=1}^{N} (M_{\alpha}f_j)^{r} \right)^{1/r}, 
\left( \sum_{j=1}^{N} |f_j|^{r} \right)^{1/r} \right) : N \in \mathbb{N}, \{f_j \}_{j=1}^{N} \subset L^{\infty}_{comp} \right\},$$
where $L^{\infty}_{comp}$ denotes the set of bounded functions with compact support on $\mathbb{H}^{n}$.

Given $1 < p_0 < p_{-}$ fixed, let $q_0$ be defined by $\frac{1}{q_0} := \frac{1}{p_0} - \frac{\alpha}{Q}$. From Proposition \ref{fractional max}, Lemma \ref{prop Apq} and following the proof of \cite[Theorem 3.23]{Cruz} (considering there $\mathbb{H}^{n}$ instead 
of $\mathbb{R}^{n}$) we have, by Lemma \ref{prop Apq}-(iv), that there exists an universal constant $C > 0$ such that for any $(F, G) \in \mathcal{F}$ and any $\omega \in \mathcal{A}_1$ 
\begin{equation} \label{weighted fractional ineq}
\int_{\mathbb{H}^{n}} [F(z)]^{q_0} \omega(z) \, dz \leq C \left( \int_{\mathbb{H}^{n}} [G(z)]^{p_0} [\omega(z)]^{p_0/q_0} \, dz 
\right)^{q_0/p_0}.
\end{equation}
On the other hand, by Lemma \ref{potencia s} and Proposition \ref{norma equivalente}, there exists an universal constant $C > 0$ such that
\begin{equation} \label{norma}
\| F \|^{q_0}_{L^{q(\cdot)}} = \| F^{q_0} \|_{L^{q(\cdot)/q_{0}}}
\leq C \sup_{\| g \|_{L^{(q(\cdot)/q_0)'}} \leq 1} \int_{\mathbb{H}^{n}} \left| [F(z)]^{q_0} g(z) \right| dz.
\end{equation}

Let $\mathcal{R}$ be the operator defined on $L^{(q(\cdot)/q_0)'}(\mathbb{H}^{n})$ by
\[
\mathcal{R}g(z) = \sum_{k=0}^{\infty} \frac{M^{k}g(z)}{2^{k} \| M \|_{L^{(q(\cdot)/q_0)'}}},
\]
where, for $k \geq 1$, $M^{k}$ denotes $k$ iterations of the Hardy-Littlewood maximal operator $M$, $M^{0} = M$, and 
$\| M \|_{L^{(q(\cdot)/q_0)'}}$ is the operator norm of the maximal operator $M$ on 
$L^{(q(\cdot)/q_0)'}$. The well definition of the operator $\mathcal{R}$ follows from \cite[Theorem 1.7]{Adam}. 
Now, it is clear that:

$(i)$ if $g$ is non-negative, $g(z) \leq \mathcal{R}g(z)$ a.e. $z \in \mathbb{H}^{n}$;

$(ii)$ $\| \mathcal{R}g \|_{L^{(q(\cdot)/q_0)'}} \leq 2 \| g \|_{L^{(q(\cdot)/q_0)'}}$; 

$(iii)$ $\mathcal{R}g \in \mathcal{A}_1$ with $[\mathcal{R}g]_{\mathcal{A}_1} \leq 2 \| M \|_{L^{(q(\cdot)/q_0)'}}$.
\\
Since $F$ is non-negative, we can take the supremum in (\ref{norma}) over those non-negative $g$ only. For any fixed non-negative 
$g \in L^{(q(\cdot)/q_0)'}$, by $(i)$ above we have that
\begin{equation} \label{int g}
\int [F(z)]^{q_0} g(z) dz \leq \int [F(z)]^{q_0} (\mathcal{R}g)(z) dz.
\end{equation}
Then $(iii)$ and (\ref{weighted fractional ineq}), and H\"older's inequality yield
\begin{equation} \label{int Rg}
\int [F(z)]^{q_0} (\mathcal{R}g)(z) dz \leq C \left( \int [G(z)]^{p_0} [(\mathcal{R}g)(z)]^{p_0 / q_0} dz \right)^{q_0/p_0} 
\end{equation}
\[
\leq C \| G^{p_0} \|_{L^{p(\cdot)/p_0}}^{q_0/p_0} 
\|(\mathcal{R}g)^{p_0/q_0} \|_{L^{(p(\cdot)/p_0)'}}^{q_0/p_0}
\]
\[
= C \| G \|^{q_0}_{L^{p(\cdot)}} 
\|\mathcal{R}g \|_{L^{\frac{p_0}{q_0} \left(\frac{p(\cdot)}{p_0} \right)'}}
\]
since $\frac{1}{p(\cdot)} - \frac{1}{q(\cdot)} = \frac{1}{p_0} - \frac{1}{q_0}$, we have 
$\frac{p_0}{q_0} \left( \frac{p(\cdot)}{p_0} \right)' = \left( \frac{q(\cdot)}{q_0} \right)'$, so
\[
= C \| G \|^{q_0}_{L^{p(\cdot)}} \| \mathcal{R}g \|_{L^{(q(\cdot)/q_0)'}}
\]
now, $(ii)$ gives
\[
\leq C \| G \|^{q_0}_{L^{p(\cdot)}} \| g \|_{L^{(q(\cdot)/q_0)'}}.
\]
Thus, (\ref{int g}) and (\ref{int Rg}) lead to
\begin{equation} \label{norma2}
\int [F(z)]^{q_0} g(z) dz \leq C \| G \|^{q_0}_{L^{p(\cdot)}},
\end{equation}
for all non-negative $g$ with $\| g \|_{L^{(q(\cdot)/q_0)'}} \leq 1$. Then, (\ref{norma}) and (\ref{norma2}) give 
(\ref{maximal fract ineq}) for all finite sequences $\{f_j \}_{j=1}^{K} \subset L^{\infty}_{comp}$. Finally, by passing to the limit, we obtain (\ref{maximal fract ineq}) for all infinite sequences $\{f_j \}_{j=1}^{\infty} \subset L^{\infty}_{comp}$.
\end{proof}

\section{Variable Hardy spaces on $\mathbb{H}^{n}$}

We recall some terminologies and notations from the study of maximal functions used in \cite{Fang}. Given $L \in \mathbb{N}$, define 
\[
\mathcal{F}_{L}=\left\{ \varphi \in \mathcal{S}(\mathbb{H}^{n}): \| \varphi \|_{\mathcal{S}(\mathbb{H}^{n}), \, L}  \leq 1\right\}.
\]
For any $f \in \mathcal{S}'(\mathbb{H}^{n})$, the grand maximal function of $f$ is given by 
\[
\mathcal{M}_L f(z)=\sup\limits_{t>0}\sup\limits_{\phi \in \mathcal{F}_{L}}\left\vert \left( f \ast \phi_t \right)(z) \right\vert,
\]
where $\phi_t(z) = t^{-2n-2} \phi(t^{-1} \cdot z)$. 

\begin{definition} \label{Dp def} Given an exponent function $p(\cdot):\mathbb{H}^{n} \to ( 0,\infty)$ with $0 < p_{-} \leq p_{+} < \infty$, we define the integer $\mathcal{D}_{p(\cdot)}$ by
\[
\mathcal{D}_{p(\cdot)} := \min \{ k \in \mathbb{N} \cup \{0\} : (2n+k+3) p_{-} > 2n+2 \}.
\]
For $L \geq \mathcal{D}_{p(\cdot)} + Q + 3$, define 
the variable Hardy space $H^{p(\cdot)}(\mathbb{H}^{n})$ to be the collection of $f \in \mathcal{S}'(\mathbb{H}^{n})$ such that
$\| \mathcal{M}_L f \|_{L^{p(\cdot)}(\mathbb{H}^{n})} < \infty$. Then, the "norm" on the space $H^{p(\cdot)}(\mathbb{H}^{n})$ is taken to be
$\| f \|_{H^{p(\cdot)}} := \| \mathcal{M}_L f \|_{L^{p(\cdot)}}$.
\end{definition}

\begin{definition} \label{atom def} Let $p(\cdot):\mathbb{H}^{n} \to ( 0,\infty)$, $0 < p_{-} \leq p_{+} < \infty $, and $p_{0} > 1$. 
Fix an integer $D \geq \mathcal{D}_{p(\cdot)}$. A measurable function $a(\cdot)$ on $\mathbb{H}^{n}$ is called a $(p(\cdot), p_{0}, D)$ - atom centered at a $\rho$ - ball $B=B(z_0, \delta)$ if it satisfies the following conditions:

$a_{1})$ $\supp ( a ) \subset B$,

$a_{2})$ $\| a \|_{L^{p_{0}}(\mathbb{H}^{n})} \leq 
\displaystyle{\frac{| B |^{\frac{1}{p_{0}}}}{\| \chi _{B} \|_{L^{p(\cdot)}(\mathbb{H}^{n})}}}$,

$a_{3})$ $\displaystyle{\int_{\mathbb{H}^{n}}} a(z) \, z^{I} \, dz = 0$ for all multiindex $I$ such that $d(I) \leq D$.
\end{definition}

Indeed, every $(p(\cdot), p_{0}, D)$ - atom $a(\cdot)$ belongs to $H^{p(\cdot)}(\mathbb{H}^{n})$. Moreover, there exists an universal constant
$C > 0$ such that $\| a \|_{H^{p(\cdot)}} \leq C$ for all $(p(\cdot), p_{0}, D)$ - atom $a(\cdot)$.

\begin{remark} \label{atomo trasladado}
It is easy to check that if $a(\cdot)$ is a $(p(\cdot), p_{0}, D)$ - atom centered at the ball $B(z_0, \delta)$, then the function 
$a_{z_0}(\cdot) := a(z_0 \cdot (\cdot))$ is a $(p(\cdot), p_{0}, D)$ - atom centered at the ball $B(e, \delta)$.
\end{remark}

\begin{definition} Let $p(\cdot):\mathbb{H}^{n} \to ( 0,\infty)$ be an exponent function such that $0 < p_{-} \leq p_{+} < \infty$.
Given a sequence of nonnegative numbers $\{ \lambda_j \}_{j=1}^{\infty}$ and a family of $\rho$ - balls $\{ B_j \}_{j=1}^{\infty}$, we define
\begin{equation} \label{cantidad A}
\mathcal{A} \left( \{ \lambda_j \}_{j=1}^{\infty}, \{ B_j \}_{j=1}^{\infty}, p(\cdot) \right) := 
\left\| \left\{ \sum_{j=1}^{\infty} \left( \frac{\lambda_j  \chi_{B_j}}{\| \chi_{B_j} \|_{L^{p(\cdot)}}} \right)^{\underline{p}} 
\right\}^{1/\underline{p}} \right\|_{L^{p(\cdot)}}.
\end{equation}
\end{definition}

To get our main results we need the following version of the atomic decomposition for $H^{p(\cdot)}(\mathbb{H}^{n})$ obtained in \cite{Fang}.

\begin{theorem} \label{atomic decomp}
Let $1 < p_0 < \infty$, $p(\cdot) \in \mathcal{P}^{\log}(\mathbb{H}^{n})$ with $0 < p_{-} \leq p_{+} < \infty$. Then, for every $f \in H^{p(\cdot)}(\mathbb{H}^{n}) \cap L^{p_0}(\mathbb{H}^{n})$ and
every integer $D \geq \mathcal{D}_{p(\cdot)}$ fixed, there exist a sequence of nonnegative numbers 
$\{ \lambda_j \}_{j=1}^{\infty}$, a sequence of $\rho$ - balls $\{ B_j \}_{j=1}^{\infty}$ with the bounded intersection property and 
$(p(\cdot), p_0, D)$ - atoms $a_j$ supported on $B_j$ such that $f = \displaystyle{\sum_{j=1}^{\infty} \lambda_j a_j}$ converges 
in $L^{p_0}(\mathbb{H}^{n})$ and
\begin{equation} \label{Hp norm atomic}
\mathcal{A} \left( \{ \lambda_j \}_{j=1}^{\infty}, \{ B_j \}_{j=1}^{\infty}, p(\cdot) \right) 
\lesssim \| f \|_{H^{p(\cdot)}(\mathbb{H}^{n})},
\end{equation}
where the implicit constant in (\ref{Hp norm atomic}) is independent of $\{ \lambda_j \}_{j=1}^{\infty}$, $\{ B_j \}_{j=1}^{\infty}$, and 
$f$.
\end{theorem}

\begin{proof}
The existence of a such atomic decomposition as well as the validity of (\ref{Hp norm atomic}) are guaranteed by 
\cite[Theorem 4.4, see p. 261 - Part 2]{Fang}. Its construction is analogous to that given for Hardy spaces on homogeneous groups; which 
in turn is similar to the construction on Euclidean spaces (see \cite{Stein3}). So, by adapting the proof of \cite[Theorem 3.1]{Rocha2} to our setting, and taking into account the atomic decomposition in \cite[see p. 97-102]{Folland}, we get the convergence of the atomic series 
to $f$ in $L^{p_0}(\mathbb{H}^{n})$.
\end{proof}

\begin{proposition} \label{dense set}
Let $1 < p_0 < \infty$ and $p(\cdot) \in \mathcal{P}^{\log}(\mathbb{H}^{n})$ with $0 < p_{-} \leq p_{+} < \infty$. Then
$H^{p(\cdot)}(\mathbb{H}^{n}) \cap L^{p_0}(\mathbb{H}^{n}) \subset H^{p(\cdot)}(\mathbb{H}^{n})$ densely.
\end{proposition}

\begin{proof}
The proof is similar to that given in \cite[see p. 3693]{Nakai}.
\end{proof}

We conclude this section with two results concerning to the amount defined by (\ref{cantidad A}).

\begin{lemma} \label{ineq p star}
Let $p(\cdot) : \mathbb{H}^{n} \to (0, \infty)$ be an exponent function with $0 < p_{-} \leq p_{+} < \infty$ and let $\{ B_j \}$ be a family of $\rho$ - balls which satisfies the bounded intersection property. If $0 < p_{*} < \underline{p}$, then
\[
\left\| \left\{ \sum_j \left( \frac{\lambda_j \chi_{B_j}}{\| \chi_{B_j} \|_{L^{p(\cdot)}}} \right)^{p_{*}} \right\}^{1/p_{*}}
\right\|_{L^{p(\cdot)}} \approx \mathcal{A} \left( \{ \lambda_j \}_{j=1}^{\infty}, \{ B_j \}_{j=1}^{\infty}, p(\cdot) \right)
\]
for any sequence of nonnegative numbers $\{ \lambda_j \}_{j=1}^{\infty}$.
\end{lemma}

\begin{proof} The embedding $\ell^{p_{*}} \subset \ell^{\underline{p}}$ implies that
\[
\mathcal{A} \left( \{ \lambda_j \}_{j=1}^{\infty}, \{ B_j \}_{j=1}^{\infty}, p(\cdot) \right) \leq
\left\| \left\{ \sum_j \left( \frac{\lambda_j \chi_{B_j}}{\| \chi_{B_j} \|_{L^{p(\cdot)}}} \right)^{p_{*}} \right\}^{1/p_{*}}
\right\|_{L^{p(\cdot)}}.
\]
On the other hand, there exists $N \in \mathbb{N}$ such that $0 < \underline{p}/N < p_{*}$ and since 
$\ell^{\underline{p}/N} \subset \ell^{p_{*}}$ embed continuously, it follows that
\[
\left\| \left\{ \sum_j \left( \frac{\lambda_j \chi_{B_j}}{\| \chi_{B_j} \|_{L^{p(\cdot)}}} \right)^{p_{*}} \right\}^{1/p_{*}}
\right\|_{L^{p(\cdot)}} \leq
\left\| \left\{ \sum_j \left( \frac{\lambda_j \chi_{B_j}}{\| \chi_{B_j} \|_{L^{p(\cdot)}}} \right)^{\underline{p}/N} 
\right\}^{N/\underline{p}} \right\|_{L^{p(\cdot)}},
\]
the bounded intersection property of the family $\{ B_j \}$ and \cite[1.1.4. (c), p. 12]{Grafakos} give
\[
\lesssim \left\| \left\{ \sum_j \left( \frac{\lambda_j \chi_{B_j}}{\| \chi_{B_j} \|_{L^{p(\cdot)}}} \right)^{\underline{p}} 
\right\}^{1/\underline{p}} \right\|_{L^{p(\cdot)}} = \mathcal{A} \left( \{ \lambda_j \}_{j=1}^{\infty}, \{ B_j \}_{j=1}^{\infty}, p(\cdot) \right).
\]
This finishes the proof.
\end{proof}

\begin{proposition} \label{Estimate_qp}
Let $0 < \alpha < Q$ and let $p(\cdot) : \mathbb{H}^{n} \to (0, \infty)$ such that $p(\cdot )\in \mathcal{P}^{\log}(\mathbb{H}^{n})$ and 
$0 < p_{-} \leq p_{+} < \frac{Q}{\alpha}$. If $\frac{1}{q(\cdot)} := \frac{1}{p(\cdot)} - \frac{\alpha}{Q}$, then
\[
\mathcal{A} \left( \{ \lambda_j \}_{j=1}^{\infty}, \{ B_j \}_{j=1}^{\infty}, q(\cdot) \right) \lesssim
\mathcal{A} \left( \{ \lambda_j \}_{j=1}^{\infty}, \{ B_j \}_{j=1}^{\infty}, p(\cdot) \right)
\]
for any sequence of nonnegative numbers $\{ \lambda_j \}_{j=1}^{\infty}$ and any family of $\rho$ - balls $\{ B_j \}_{j=1}^{\infty}$ of 
$\mathbb{H}^{n}$.
\end{proposition}

\begin{proof} Since $\ell^{\underline{p}} \subset \ell^{\underline{q}}$ embed continuously, we have 
\[
\mathcal{A} \left( \{ \lambda_j \}_{j=1}^{\infty}, \{ B_j \}_{j=1}^{\infty}, q(\cdot) \right) =
\left\| \left\{ \sum_{j} \left( \frac{\lambda_j \chi_{B_j}}{\| \chi_{B_j} \|_{L^{q(\cdot)}}} \right)^{\underline{q}} 
\right\}^{1/\underline{q}} \right\|_{L^{q(\cdot)}}
\]
\[
\lesssim \left\| \left\{ \sum_{j} \left( \frac{ \lambda_j \chi_{B_j}}{\| \chi_{B_j} \|_{L^{q(\cdot)}}} \right)^{\underline{p}} 
\right\}^{1/\underline{p}} \right\|_{L^{q(\cdot)}},
\]
\cite[Lemma 4.1]{Fang} gives $\| \chi_B \|_{L^{q(\cdot)}} \approx |B|^{-\alpha/Q} \| \chi_B \|_{L^{p(\cdot)}}$ for every 
$\rho$ - ball $B$ of $\mathbb{H}^{n}$, so
\[
\lesssim \left\| \left\{ \sum_{j} \left( \frac{\lambda_j |B_j|^{\alpha/Q} \chi_{B_j}}{\| \chi_{B_j} \|_{L^{p(\cdot)}}} 
\right)^{\underline{p}} \right\}^{1/\underline{p}} \right\|_{L^{q(\cdot)}},
\]
now it is easy to check that 
$|B_j|^{\alpha/Q} \chi_{B_j} (z) \leq M_{\frac{\alpha \underline{p}}{2}}(\chi_{B_j})^{\frac{2}{\underline{p}}}(z)$ for all $j$, then
\[
\lesssim \left\| \left\{ \sum_{j} \left( \frac{\lambda_j \, M_{\frac{\alpha \underline{p}}{2}}(\chi_{B_j})^{\frac{2}{\underline{p}}}}{\| \chi_{B_j} \|_{L^{p(\cdot)}}} \right)^{\underline{p}} \right\}^{1/\underline{p}} \right\|_{L^{q(\cdot)}}
\]
Lemma \ref{potencia s}-(iv) gives
\[ 
= \left\| \left\{ \sum_{j} \left( \frac{\lambda_{j}^{\underline{p}} \, M_{\frac{\alpha \underline{p}}{2}}(\chi_{B_j})^{2}}{\| \chi_{B_j} \|^{\underline{p}}_{L^{p(\cdot)}}} \right) \right\}^{1/2}
\right\|^{2/\underline{p}}_{L^{2q(\cdot)/\underline{p}}},
\]
by applying Theorem \ref{Feff-Stein ineq} we obtain
\[
\lesssim \left\| \left\{ \sum_{j} \left( \frac{\lambda_{j}^{\underline{p}} \, \chi_{B_j}}{\| \chi_{B_j} \|^{\underline{p}}_{L^{p(\cdot)}}} \right) \right\}^{1/2} \right\|^{2/\underline{p}}_{L^{2p(\cdot)/\underline{p}}} =
\mathcal{A} \left( \{ \lambda_j \}_{j=1}^{\infty}, \{ B_j \}_{j=1}^{\infty}, p(\cdot) \right).
\]
This completes the proof.
\end{proof}

\section{Main results} 

The convolution kernels we shall be considering are the introduced by Folland and Stein \cite{Folland} (see Ch. 6 and Remark 6.12). 
Suppose $0 \leq \alpha < Q$ and $N \in \mathbb{N}$. For $0 < \alpha < Q$  a kernel of type $(\alpha, N)$ is a function $K_{\alpha}$ of 
class $C^{N}$ on $\mathbb{H}^{n} \setminus \{ e \}$, which satisfies
\begin{equation} \label{decay}
\left|(\widetilde{X}^{I}K_{\alpha})(z) \right| \lesssim \rho(z)^{\alpha - Q - d(I)} \,\,\,\, \text{for all} \,\, d(I) \leq N \,\,\, \text{and all} \,\, z \neq e.
\end{equation}
A kernel of type $(0,N)$ is a distribution $K_{0}$ on $\mathbb{H}^{n}$ which is of class $C^{N}$ on $\mathbb{H}^{n} \setminus \{ e \}$, satisfies (\ref{decay}) with $\alpha = 0$, and
\begin{equation} \label{cota L2}
\|f \ast K_0 \|_{L^{2}(\mathbb{H}^{n})} \leq C \| f \|_{L^{2}(\mathbb{H}^{n})}, \,\,\, \text{for all} \,\, f \in \mathcal{S}(\mathbb{H}^{n}).
\end{equation}

\begin{remark} \label{operator Ta}
If $0 < \alpha < Q$ and $K_{\alpha}$ is a kernel of type $(\alpha, N)$, from \cite[Proposition 6.2]{Folland}, it follows that the operator 
$T_{\alpha} : f \to f \ast K_{\alpha}$ is bounded from $L^{p_0}(\mathbb{H}^{n})$ to $L^{q_0}(\mathbb{H}^{n})$ for 
$1 < p_0 < \frac{Q}{\alpha}$ and $\frac{1}{q_0} = \frac{1}{p_0} - \frac{\alpha}{Q}$.
\end{remark}

\begin{remark} \label{operator T0}
Given a kernel $K_0$ of type $(0,N)$, by (\ref{cota L2}), it follows that the operator $U_0 : f \to f \ast K_0$, 
$f \in \mathcal{S}(\mathbb{H}^{n})$, can be extended to a bounded operator on $L^{2}(\mathbb{H}^{n})$, a such extension is unique. We denote this extension by $T_0$. Now, it is easy to check that if $a(\cdot) \in L^{2}(\mathbb{H}^{n})$ and their support is contained 
in the $\rho$ - ball $B(z_0, \delta)$, then $T_{0}a(z) = (a \ast K_{0})(z)$ a.e. $z \notin B(z_0, 2\delta)$.
\end{remark}

In the sequel, given a kernel $K_{\alpha}$ of type $(\alpha, N)$ with $0 \leq \alpha < Q$, we consider the operator $T_{\alpha}$ defined by
\begin{equation} \label{operator T alpha}
T_{\alpha} = \left\{ \begin{array}{c}
                      \text{right convolution operator by} \, K_{\alpha}, \,\, \text{if} \,\, 0 < \alpha < Q \\
											\text{extension of the operator} \,\, U_0 \,\, \text{on} \,\, L^{2}(\mathbb{H}^{n}), \,\, \text{if} \,\, \alpha = 0
											\end{array} \right..
\end{equation}

\begin{theorem} \label{Hp-Lq}
Let $N \in \mathbb{N}$, $0 \leq \alpha < Q$, and $p(\cdot) \in \mathcal{P}^{\log}(\mathbb{H}^{n})$ with 
$\frac{Q}{Q+N} < p_{-} \leq p_{+} < \frac{Q}{\alpha}$. If $\frac{1}{q(\cdot)} = \frac{1}{p(\cdot)} - \frac{\alpha}{Q}$, then 
the operator $T_{\alpha}$ defined by (\ref{operator T alpha}) can be extended to a bounded operator from $H^{p(\cdot)}(\mathbb{H}^{n})$ 
into $L^{q(\cdot)}(\mathbb{H}^{n})$.
\end{theorem}

\begin{proof} We recall that $Q = 2n+2$. The condition $\frac{Q}{Q+N} < p_{-}$ implies that $N-1 \geq \mathcal{D}_{p(\cdot)}$. So, 
given $f \in H^{p(\cdot)}(\mathbb{H}^{n}) \cap L^{p_0} (\mathbb{H}^{n})$ (with $p_0 > 1$), by Theorem \ref{atomic decomp} with $D=N-1$, there exist a sequence of nonnegative numbers $\{ \lambda_j \}_{j=1}^{\infty}$, a sequence of 
$\rho$ - balls $\{ B_j \}_{j=1}^{\infty}$ and $(p(\cdot), p_0, N-1)$ atoms $a_j$ supported on $B_j$ such that 
$f = \displaystyle{\sum_{j=1}^{\infty} \lambda_j a_j}$ converges in $L^{p_0}(\mathbb{H}^{n})$ and
\begin{equation} \label{atomic Hp norm}
\mathcal{A} \left( \{ \lambda_j \}_{j=1}^{\infty}, \{ B_j \}_{j=1}^{\infty}, p(\cdot) \right) 
\lesssim \| f \|_{H^{p(\cdot)}(\mathbb{H}^{n})}.
\end{equation}
If $0 < \alpha < Q$, we take $\max\{1, p_{+} \} < p_0 < \frac{Q}{\alpha}$. If $\alpha = 0$, we take $p_0 = 2$. Then, by Remark 
\ref{operator Ta}, the operator $T_{\alpha}$ is bounded from $L^{p_0}(\mathbb{H}^{n})$ to $L^{q_0}(\mathbb{H}^{n})$ for $1 < p_0 < \frac{Q}{\alpha}$ and $\frac{1}{q_0} = \frac{1}{p_0} - \frac{\alpha}{Q}$. For the case $\alpha = 0$, by Remark \ref{operator T0}, the operator $T_0$ is bounded on $L^{p_0}(\mathbb{H}^{n})$ with $p_0 =2$. Since $f = \displaystyle{\sum_{j=1}^{\infty} \lambda_j a_j}$ converges in 
$L^{p_0}(\mathbb{H}^{n})$, we have
\[
|T_{\alpha} f(z)| \leq \sum_j \lambda_j |T_{\alpha} a_j(z)|, \,\,\, a.e. \,\, z \in \mathbb{H}^{n}.
\]
Let $\beta$ be the constant in \cite[Corollary 1.44]{Folland}, we observe that $\beta \geq 1$ (see \cite[p. 29]{Folland}). Then, 
for $\frac{1}{q(\cdot)} = \frac{1}{p(\cdot)} - \frac{\alpha}{Q}$
\[
\| T_{\alpha} f \|_{L^{q(\cdot)}} \leq \left\| \sum_{j} \lambda_j \chi_{2\beta^{N}B_{j}} |T_{\alpha} a_j| \right\|_{L^{q(\cdot)}}  + 
\left\| \sum_j \lambda_j \chi_{\mathbb{H}^{n} \setminus 2\beta^{N}B_{j}} |T_{\alpha} a_j| \right\|_{L^{q(\cdot)}} =: L_1 + L_2, 
\]
where $2\beta^{N}B_j$ is the $\rho$ - ball with the same center as $B_j$ but whose radius is expanded by the factor $2\beta^{N}$. This is,
if $B_j = B(z_j, \delta_j)$ then $2\beta^{N}B_j = B(z_j, 2\beta^{N} \delta_j)$.

To estimate $L_1$ we apply, for the case $0 < \alpha < Q$, Remark \ref{operator Ta} with $q_0  > \max\{ \frac{Q}{Q - \alpha}, q_{+} \}$ and 
$\frac{1}{p_0} := \frac{1}{q_0} + \frac{\alpha}{Q}$ (or Remark \ref{operator T0} with $q_0=p_0=2$, if $\alpha =0$). So,
\[
\left\| (T_{\alpha} a_j)^{q_{*}} \right\|_{L^{q_{0}/q_{*}}(2\beta^{N}B_{j})} = 
\left\| T_{\alpha} a_j \right\|_{L^{q_{0}}(2\beta^{N}B_{j})}^{q_{*}} \lesssim \left\| a_j \right\|_{L^{p_{0}}}^{q_{*}} \lesssim 
\frac{| B_j |^{\frac{q_{*}}{p_{0}}}}{\left\| \chi _{B_j }\right\|_{L^{p(\cdot)}}^{q_{*}}} \lesssim 
\frac{\left| 2\beta^{N}B_{j} \right|^{\frac{q_{*}}{q_{0}}}}{\left\| \chi_{2\beta^{N}B_{j}} \right\|_{L^{q(\cdot)/q_{*}}}},
\]
where $0 < q_{*} < \underline{q}$ is fixed and the last inequality follows from the estimate $\| \chi_B \|_{L^{q(\cdot)}} \approx |B|^{-\alpha/Q} \| \chi_B \|_{L^{p(\cdot)}}$, Lemma \ref{potencia s}-(iv), and Lemma \ref{2B} applied to the exponent $q(\cdot)/q_{*}$. Now, since $0 < q_{*} < 1$, we apply the $q_{\ast}$-inequality and Proposition \ref{b_k functions} with 
$b_j = \left( \chi_{2\beta^{N}B_j} \cdot |T_{\alpha} a_j| \right)^{q_{*}}$, 
$A_j = \left\| \chi_{2\beta^{N}B_{j}} \right\|_{L^{q(\cdot)/q_{*}}}^{-1}$ and $s= q_0/q_{*}$, to obtain
\[
L_1 \lesssim \left\| \sum_{j} \left(\lambda_j \, \chi_{2\beta^{N} B_j} \, |T_{\alpha} a_j| \right)^{q_{*}} \right\|^{1/q_{*}}_{L^{q(\cdot)/q_{*}}} \lesssim \left\| \sum_{j} \left( \frac{\lambda_j}{\left\| \chi_{2\beta^{N}B_{j}} \right\|_{L^{q(\cdot)}}} \right)^{q_{*}} 
\chi_{2\beta^{N} B_j} \right\|^{1/q_{*}}_{L^{q(\cdot)/q_{*}}}.
\]
It is easy to check that $\chi_{2 \beta^{N} B_j} \leq [M(\chi_{B_j})]^{2}$. From this inequality, Lemma \ref{2B}, Lemma \ref{potencia s}-(iv), and Theorem \ref{Feff-Stein ineq} we have
\[
L_1 \lesssim \left\| \left\{ \sum_{j} \left( \frac{\lambda_j^{q_{*}/2}}{\left\| \chi_{B_{j}} \right\|^{q_{*}/2}_{L^{q(\cdot)}}}
M(\chi_{B_j})\right)^{2}  \right\}^{1/2} \right\|^{2/q_{*}}_{L^{2q(\cdot)/q_{*}}} \lesssim
\left\| \left\{ \sum_{j} \left( \frac{ \lambda_j \chi_{B_j}}{\| \chi_{B_j} \|_{L^{q(\cdot)}}} \right)^{q_{*}} 
\right\}^{1/q_{*}} \right\|_{L^{q(\cdot)}}.
\]
Lemma \ref{ineq p star} applied to $q(\cdot)$, Proposition \ref{Estimate_qp} and (\ref{atomic Hp norm}) give
\begin{equation} \label{L1}
L_1 \lesssim \mathcal{A}\left( \{ \lambda_j \}_{j=1}^{\infty}, \{ B_j \}_{j=1}^{\infty}, q(\cdot) \right)
\lesssim \mathcal{A}\left( \{ \lambda_j \}_{j=1}^{\infty}, \{ B_j \}_{j=1}^{\infty}, p(\cdot) \right) 
\lesssim \| f \|_{H^{p(\cdot)}}.
\end{equation}

Now, we proceed to estimate $L_2$. For them, we first consider a $(p(\cdot), p_0, N-1)$ - atom $a(\cdot)$ supported on the $\rho$ - ball 
$B = B(z_0, \delta)$. Then, by Remark \ref{operator Ta} (for $0 < \alpha < Q$) or Remark \ref{operator T0} (for $\alpha = 0$) and 
Remark \ref{cambio de centro}, we have
\[
T_{\alpha} a(z) = \int_{B(z_0, \delta)} a(w) K_{\alpha}(w^{-1} \cdot z) \, dw = \int_{B(e, \delta)} a(z_0 \cdot u) 
K_{\alpha}(u^{-1} \cdot z_{0}^{-1} \cdot z) \, du, 
\]
for every $z \notin B(z_0, 2\beta^{N}\delta)$. By Remark \ref{atomo trasladado}, it follows for $z \notin B(z_0, 2\beta^{N}\delta)$ that
\begin{equation} \label{Taz}
T_{\alpha} a(z) = \int_{B(e, \delta)} a(z_0 \cdot u) \left[K_{\alpha}(u^{-1} \cdot z_{0}^{-1} \cdot z) - q(u^{-1}) \right] \, du, 
\end{equation} 
where $u \to q(u^{-1})$ is the right Taylor polynomial of the function $u \to K_{\alpha}(u^{-1} \cdot z_{0}^{-1} \cdot z)$ at 
$e$ of homogeneous degree $N-1$. Then by the right-invariant version of the Taylor inequality in \cite[Corollary 1.44]{Folland},
\begin{equation} \label{Taylor ineq}
\left| K_{\alpha}(u^{-1} \cdot z_{0}^{-1} \cdot z) - q(u^{-1}) \right| \lesssim \rho(u)^{N} 
\sup_{\rho(v) \leq \beta^{N}\rho(u), \, d(I)=N} \left|(\widetilde{X}^{I}K_{\alpha})(v \cdot z_{0}^{-1} \cdot z) \right|.
\end{equation}
Now, for $u \in B(e, \delta)$, $z_{0}^{-1} \cdot z \notin B(e, 2\beta^{N}\delta)$ and $\rho(v) \leq \beta^{N} \rho(u)$, we have
$\rho(z_{0}^{-1} \cdot z) \geq 2\rho(v)$ and hence $\rho(v \cdot z_{0}^{-1} \cdot z) \geq \rho(z_{0}^{-1} \cdot z)/2$, then 
by (\ref{Taylor ineq}) and (\ref{decay}) we get
\[
\left| K_{\alpha}(u^{-1} \cdot z_{0}^{-1} \cdot z) - q(u^{-1}) \right| \lesssim \delta^{N} \rho(z_{0}^{-1} \cdot z)^{\alpha-Q-N}.
\]
This estimate and (\ref{Taz}) lead to
\begin{eqnarray*}
|T_{\alpha}a(z)| \lesssim \delta^{N} \rho(z_{0}^{-1} \cdot z)^{\alpha-Q-N} \| a \|_{L^{1}} &\lesssim& \delta^{N} \rho(z_{0}^{-1} \cdot z)^{\alpha-Q-N} |B|^{1-\frac{1}{p_0}} \| a \|_{L^{p_0}} \\
&\lesssim& \frac{\delta^{N+Q}}{\| \chi_{B} \|_{L^{p(\cdot)}}} \rho(z_{0}^{-1} \cdot z)^{\alpha-Q-N} \\
&\lesssim& \frac{\left( M_{\frac{\alpha Q}{Q+N}}(\chi_{B})(z) \right)^{\frac{Q+N}{Q}}}{\| \chi_{B} \|_{L^{p(\cdot)}}}, \,\,\,\, 
\forall z \notin 2\beta^{N}B.
\end{eqnarray*}
So, for every $j \in \mathbb{N}$, we have that
\begin{equation} \label{L2 estimate}
|T_{\alpha}a_j(z)| \lesssim \frac{\left( M_{\frac{\alpha Q}{Q+N}}(\chi_{B_j})(z) \right)^{\frac{Q+N}{Q}}}{\| \chi_{B_j} \|_{L^{p(\cdot)}}},
\,\,\,\, \text{for all} \,\, z \notin 2\beta^{N}B_j.
\end{equation}
From (\ref{L2 estimate}) follows that
\[
L_2 \lesssim \left\| \sum_j \lambda_j 
\frac{\left( M_{\frac{\alpha Q}{Q+N}}(\chi_{B_j})(\cdot) \right)^{\frac{Q+N}{Q}}}{\| \chi_{B_j} \|_{L^{p(\cdot)}}} \right\|_{L^{q(\cdot)}}
= \left\| \left\{ \sum_j \lambda_j 
\frac{\left( M_{\frac{\alpha Q}{Q+N}}(\chi_{B_j})(\cdot) \right)^{\frac{Q+N}{Q}}}{\| \chi_{B_j} \|_{L^{p(\cdot)}}} 
\right\}^{\frac{Q}{Q+N}} \right\|_{L^{\frac{Q+N}{Q}q(\cdot)}}^{\frac{Q+N}{Q}}.
\]
Since $1 < \frac{Q+N}{Q} p_{-} \leq \frac{Q+N}{Q} p_{+} < \frac{Q+N}{\alpha}$, Theorem \ref{Feff-Stein ineq} gives
\begin{equation} \label{L2}
L_2 \lesssim 
\left\| \left\{ \sum_j \lambda_j 
\frac{\chi_{B_j}}{\| \chi_{B_j} \|_{L^{p(\cdot)}}} \right\}^{\frac{Q}{Q+N}} \right\|_{L^{\frac{Q+N}{Q}p(\cdot)}}^{\frac{Q+N}{Q}} =
\left\| \sum_j \lambda_j \frac{\chi_{B_j}}{\| \chi_{B_j} \|_{L^{p(\cdot)}}} \right\|_{L^{p(\cdot)}}
\end{equation}
\[
\lesssim
\left\| \left\{ \sum_j \left( \lambda_j \frac{\chi_{B_j}}{\| \chi_{B_j} \|_{L^{p(\cdot)}}} \right)^{\underline{p}}\right\}^{1/\underline{p}} 
\right\|_{L^{p(\cdot)}} = \mathcal{A}\left( \{ \lambda_j \}_{j=1}^{\infty}, \{ B_j \}_{j=1}^{\infty}, p(\cdot) \right) 
\lesssim \| f \|_{H^{p(\cdot)}}.
\]
Finally, (\ref{L1}) and (\ref{L2}) allow us to conclude that
\[
\| T_{\alpha} f \|_{L^{q(\cdot)}(\mathbb{H}^{n})} \lesssim \| f \|_{H^{p(\cdot)}(\mathbb{H}^{n})},
\]
for all $f \in H^{p(\cdot)}(\mathbb{H}^{n}) \cap L^{p_0}(\mathbb{H}^{n})$, so the theorem follows from Proposition \ref{dense set}.
\end{proof}

\begin{theorem} \label{Hp-Hq}
Let $N \in \mathbb{N}$, $0 \leq \alpha < Q$, and $p(\cdot) \in \mathcal{P}^{\log}(\mathbb{H}^{n})$ with 
$\frac{Q}{Q+N} < p_{-} \leq p_{+} < \frac{Q}{\alpha}$. If $\frac{1}{q(\cdot)} = \frac{1}{p(\cdot)} - \frac{\alpha}{Q}$, then 
the operator $T_{\alpha}$ given by (\ref{operator T alpha}) can be extended to a bounded operator from $H^{p(\cdot)}(\mathbb{H}^{n})$ 
into $H^{q(\cdot)}(\mathbb{H}^{n})$.
\end{theorem}

\begin{proof}
The inequality $\frac{Q}{Q+N} < p_{-}$ implies that $N - 1 \geq \mathcal{D}_{p(\cdot)}$, and thus  
given $f \in H^{p(\cdot)}(\mathbb{H}^{n}) \cap L^{p_0}(\mathbb{H}^{n})$ (with $p_0 > 1$ chosen as in Theorem \ref{Hp-Lq}), by 
Theorem \ref{atomic decomp}, there exist a sequence of nonnegative numbers $\{ \lambda_j \}_{j=1}^{\infty}$, a sequence of 
$\rho$ - balls $\{ B_j \}_{j=1}^{\infty}$ and $(p(\cdot), p_0, N-1)$ atoms $a_j$ supported on $B_j$ such that 
$f = \displaystyle{\sum_{j=1}^{\infty} \lambda_j a_j}$ converges in $L^{p_0}(\mathbb{H}^{n})$ and
$\mathcal{A} \left( \{ \lambda_j \}_{j=1}^{\infty}, \{ B_j \}_{j=1}^{\infty}, p(\cdot) \right) 
\lesssim \| f \|_{H^{p(\cdot)}(\mathbb{H}^{n})}$. Since $T_{\alpha}$ is bounded from $L^{p_0}(\mathbb{H}^{n})$ into 
$L^{q_0}(\mathbb{H}^{n})$ and $H^{q_0}(\mathbb{H}^{n}) \equiv L^{q_0}(\mathbb{H}^{n})$ with comparable norms, it follows for every $L \geq 0$ that
\[
\mathcal{M}_{L}(T_{\alpha} f)(z) \leq \sum_{j=1}^{\infty} \lambda_j \mathcal{M}_{L}(T_{\alpha} a_j)(z), \,\,\, a.e. \,\, z \in \mathbb{H}^{n}.
\]
Let $\beta \geq 1$ be the constant as in Theorem \ref{Hp-Lq}. Then, for $\frac{1}{q(\cdot)} = \frac{1}{p(\cdot)} - \frac{\alpha}{Q}$ and
$L \geq {D}_{q(\cdot)} + 2Q + 3 + N - \alpha$ (see Definition \ref{Dp def}), 
\[
\| T_{\alpha} f \|_{H^{q(\cdot)}} = \| \mathcal{M}_L(T_{\alpha} f) \|_{L^{q(\cdot)}} \leq \left\| \sum_{j} \lambda_j \chi_{2\beta^{N}B_{j}} 
\mathcal{M}_L(T_{\alpha} a_j) \right\|_{L^{q(\cdot)}}  
\]
\[
+ 
\left\| \sum_j \lambda_j \chi_{\mathbb{H}^{n} \setminus 2\beta^{N}B_{j}} \mathcal{M}_L(T_{\alpha} a_j) \right\|_{L^{q(\cdot)}} =: 
J_1 + J_2, 
\]
To estimate $J_1$, we observe that for $0 < q_{*} < \underline{q}$
\[
\left\| [\mathcal{M}_L (T_{\alpha} a_j)]^{q_{*}} \right\|_{L^{q_{0}/q_{*}}(2\beta^{N}B_{j})} = 
\left\| \mathcal{M}_L (T_{\alpha} a_j) \right\|_{L^{q_{0}}(2\beta^{N}B_{j})}^{q_{*}} \lesssim 
\left\| T_{\alpha} a_j \right\|_{L^{q_{0}}}^{q_{*}} 
\]
\[
\lesssim \left\| a_j \right\|_{L^{p_{0}}}^{q_{*}}
\lesssim 
\frac{| B_j |^{\frac{q_{*}}{p_{0}}}}{\left\| \chi _{B_j }\right\|_{L^{p(\cdot)}}^{q_{*}}} \lesssim 
\frac{\left| 2\beta^{N}B_{j} \right|^{\frac{q_{*}}{q_{0}}}}{\left\| \chi_{2\beta^{N}B_{j}} \right\|_{L^{q(\cdot)/q_{*}}}}.
\]
Then, by proceeding as in the estimate of $L_1$ in Theorem \ref{Hp-Lq}, we get
\[
J_1 \lesssim \mathcal{A} \left( \{ \lambda_j \}_{j=1}^{\infty}, \{ B_j \}_{j=1}^{\infty}, p(\cdot) \right) 
\lesssim \| f \|_{H^{p(\cdot)}}.
\]

Now, we estimate $J_2$. For them, let $\phi \in \mathcal{S}(\mathbb{H}^{n})$ with $\| \phi \|_{\mathcal{S}(\mathbb{H}^{n}), \, L} \leq 1$
and let $a(\cdot)$ be a $(p(\cdot), p_0, N-1)$ - atom centered at the $\rho$ - ball $B=B(z_0, \delta)$. Then, for $z \notin 2\beta^{N} B$ and every $t > 0$, we have
\[
((T_{\alpha} a) \ast \phi_t)(z) = \int_{B(z_0, \delta)} a(w) (K_{\alpha} \ast \phi_t)(w^{-1} \cdot z) \, dw = 
\int_{B(e, \delta)} a(z_0 \cdot u) (K_{\alpha} \ast \phi_t)(u^{-1} \cdot z_{0}^{-1} \cdot z) \, du. 
\]
By Remark \ref{atomo trasladado}, it follows for $z \notin 2\beta^{N} B$ and every $t > 0$ that
\begin{equation} 
((T_{\alpha} a) \ast \phi_t)(z) = \int_{B(e, \delta)} a(z_0 \cdot u) \left[(K_{\alpha} \ast \phi_t)(u^{-1} \cdot z_{0}^{-1} \cdot z) - 
q_t(u^{-1}) \right] \, du, 
\end{equation} 
where $u \to q_t(u^{-1})$ is the right Taylor polynomial of the function $u \to (K_{\alpha} \ast \phi_t)(u^{-1} \cdot z_{0}^{-1} \cdot z)$ at 
$e$ of homogeneous degree $N-1$. Then by the right-invariant version of the Taylor inequality in \cite[Corollary 1.44]{Folland},
\[
\left| (K_{\alpha} \ast \phi_t)(u^{-1} \cdot z_{0}^{-1} \cdot z) - q_t(u^{-1}) \right| \lesssim \rho(u)^{N} 
\sup_{\rho(v) \leq \beta^{N}\rho(u), \, d(I)=N} \left|\widetilde{X}^{I}(K_{\alpha} \ast \phi_t)(v \cdot z_{0}^{-1} \cdot z) \right|.
\]
To apply \cite[Lemma 6.9]{Folland} with $r=N$ (and taking into account that $\mathbb{H}^{n}$ is stratified group, we observe that such lemma holds for $d(I) \leq r$ rather than for $|I| \leq r$), we get
\[
|\widetilde{X}^{I}(K_{\alpha} \ast \phi_t)(v \cdot w)| = |((\widetilde{X}^{I} K_{\alpha}) \ast \phi_t)(v \cdot w)|
\lesssim \rho(v \cdot w)^{\alpha - Q - d(I)}, \,\,\, \text{for} \,\,\, t>0, \, d(I) \leq N.
\]
This estimate does not depend on $t$ and $\phi \in \mathcal{S}(\mathbb{H}^{n})$ with $\| \phi \|_{\mathcal{S}(\mathbb{H}^{n}), \, L} \leq 1$ 
($L > Q + N - \alpha$). Finally, according to the ideas to estimate $L_2$ in Theorem \ref{Hp-Lq}, we obtain
\[
J_2 \lesssim \mathcal{A}\left( \{ \lambda_j \}_{j=1}^{\infty}, \{ B_j \}_{j=1}^{\infty}, p(\cdot) \right) 
\lesssim \| f \|_{H^{p(\cdot)}},
\]
for all $f \in H^{p(\cdot)}(\mathbb{H}^{n}) \cap L^{p_0}(\mathbb{H}^{n})$, and so the proof is concluded.
\end{proof}

Let $0 < \alpha < Q$, the Riesz potential $\mathcal{R}_{\alpha}$ on $\mathbb{H}^{n}$ is defined by
\begin{equation} \label{Riesz}
\mathcal{R}_{\alpha}f(z) = \int_{\mathbb{H}^{n}} f(w) \rho(w^{-1} \cdot z)^{\alpha-Q} \, dw,
\end{equation}
where $\rho(\cdot)$ is the Koranyi norm given by (\ref{Koranyi norm}). It is clear that 
$\rho(\cdot)^{\alpha - Q} \in C^{\infty}(\mathbb{H}^{n} \setminus \{e\}) = \displaystyle{\bigcap_{N \in \mathbb{N}}} C^{N}(\mathbb{H}^{n} \setminus \{e\})$ and satisfies the condition (\ref{decay}) for every $N \in \mathbb{N}$. Finally, to apply the Theorems \ref{Hp-Lq} and
\ref{Hp-Hq} with $K_{\alpha}(\cdot) = \rho(\cdot)^{\alpha-Q}$ and $0 < \alpha < Q$, we obtain the following result.

\begin{theorem}  \label{Riesz estimates}
Let $0 < \alpha < Q$, and $p(\cdot) \in \mathcal{P}^{\log}(\mathbb{H}^{n})$ with 
$0 < p_{-} \leq p_{+} < \frac{Q}{\alpha}$. If $\frac{1}{q(\cdot)} = \frac{1}{p(\cdot)} - \frac{\alpha}{Q}$, then 
the Riesz potential $\mathcal{R}_{\alpha}$ given by (\ref{Riesz}) can be extended to a bounded operator from $H^{p(\cdot)}(\mathbb{H}^{n})$ 
into $L^{q(\cdot)}(\mathbb{H}^{n})$ and from $H^{p(\cdot)}(\mathbb{H}^{n})$ into $H^{q(\cdot)}(\mathbb{H}^{n})$.
\end{theorem}

\bigskip
\address{
Departamento de Matem\'atica \\
Universidad Nacional del Sur \\
8000 Bah\'{\i}a Blanca, Buenos Aires \\
Argentina}
{pablo.rocha@uns.edu.ar}

\end{document}